\documentclass[11pt]{amsart} 
  \usepackage[margin=3.8cm]{geometry}
\numberwithin{equation}{section}
\usepackage{amsmath,amsfonts,amsthm,amssymb,amscd, verbatim,microtype,graphicx,color,multirow,booktabs, caption,tikz,tikz-cd,mathdots,bm,mathrsfs}
\usepackage{tikz-cd}
 \usepackage[pagebackref]{hyperref} 
\usetikzlibrary{positioning} 
\usepackage{mathtools}
\usepackage{etoolbox}
\usepackage{dynkin-diagrams}

\newtheorem{theorem}{Theorem}[section]
\newtheorem{lemma}[theorem]{Lemma}
\newtheorem{proposition}[theorem]{Proposition}
 \newtheorem{corollary}[theorem]{Corollary}
   
    \newtheorem{conjecture}[theorem]{Conjecture}

      \theoremstyle{definition}
     
     \newtheorem{example}[theorem]{Example}
     
     \theoremstyle{remark}
     \newtheorem{remark}[theorem]{Remark}

\newcommand{\Sym}{\mathop{\mathrm{Sym}}}
\newcommand{\Alt}{\mathop{\mathrm{Alt}}}
\newcommand{\Aut}{\mathop{\mathrm{Aut}}}

\DeclareMathOperator{\GL}{GL}
\DeclareMathOperator{\SL}{SL}
\newcommand{\AGL}{\mathop{\mathrm{AGL}}}

\newcommand{\PSU}{\mathop{\mathrm{PSU}}}
\newcommand{\PSp}{\mathop{\mathrm{PSp}}}

\newcommand{\Fi}{\mathop{\mathrm{Fi}}}

\newcommand{\Out}{\mathop{\mathrm{Out}}}

\newcommand{\dll}{\mathrm{dl^*}}
\newcommand{\dl}{\mathrm{dl}}
\newcommand{\dllMax}{\mathrm{dl^*Max}}

\DeclareMathOperator{\diag}{diag}

\newcommand{\boldL}{\mathbf{L}}
\DeclareMathOperator{\GammaL}{\Gamma L}
\DeclareMathOperator{\PGammaL}{P \Gamma L}

\DeclareMathOperator{\PSL}{PSL}
\let\L\relax
\DeclareMathOperator{\L}{L}

\newcommand{\boldU}{\mathbf{U}}
\DeclareMathOperator{\GU}{GU}
\DeclareMathOperator{\SU}{SU}

\newcommand{\boldS}{\mathbf{S}}

\DeclareMathOperator{\Sp}{Sp}

\newcommand{\boldO}{\mathbf{O}}

\DeclareMathOperator{\GammaO}{\Gamma O}
\DeclareMathOperator{\POmega}{P \Omega}
\let\O\relax
\DeclareMathOperator{\O}{O}

\newcommand{\Atlas}{\mathbb{A}\mathbb{T}\mathbb{L}\mathbb{A}\mathbb{S}}

  \definecolor{mycolor}{rgb}{0.0,0.0,0.7}
  \definecolor{myred}{rgb}{0.75,0.0,0.16} 
  \definecolor{mygreen}{rgb}{0.0,0.4,0.16} 
  \definecolor{myviolet}{rgb}{1,0,1} 
   \definecolor{mypink}{rgb}{0.9,0,0.5}

	\newcommand{\MYhref}[3][black]{\href{#2}{\color{#1}{#3}}}%

\hypersetup{
colorlinks=true,
linkcolor=true,
linktocpage=true,
pageanchor=true,
hyperindex=true
}

 \AtBeginDocument{
     \hypersetup{
  linkcolor=mycolor,
  urlcolor=mypink,
citecolor=mygreen
}
     }

\makeatletter
\@namedef{subjclassname@2020}{%
  \textup{2020} Mathematics Subject Classification}
\makeatother


\thanks{This work was supported by the Additional Funding Programme for Mathematical Sciences, delivered by EPSRC (EP/V521917/1) and the Heilbronn Institute for Mathematical Research.}

\author[Patricia Medina Capilla]{Patricia Medina Capilla}
\address{\parbox{\linewidth}{Patricia Medina Capilla, Mathematics Institute, University of Warwick\\
Coventry CV4\,7AL, United Kingdom \vspace{0.1cm}}}
\email{patricia.medina-capilla@warwick.ac.uk} 

\author[Luca Sabatini]{Luca Sabatini}
\address{\parbox{\linewidth}{Luca Sabatini, Mathematics Institute, University of Warwick\\
Coventry CV4\,7AL, United Kingdom \vspace{0.1cm}}}
\email{luca.sabatini@warwick.ac.uk, sabatini.math@gmail.com} 

\begin{document} 
\title[Maximal subgroups of almost simple and perfect groups]{On the maximal subgroups of almost simple\\and primitive perfect groups}

\maketitle 

\begin{abstract} 
We prove that, if $G$ is a finite almost simple group and $H$ is a maximal subgroup of $G$, then the $10$th term of the derived series of $H$ is perfect.
The same is true if $G$ is perfect and $H$ is core-free.
The constant $10$ is best possible.
\end{abstract}

\vspace{0.5cm}
\section{Introduction} 

One of the major areas of research in finite group theory is the investigation of the maximal subgroups of the
finite simple groups \cite{KL90}.
When possible, it is desirable to have results that are not in the form of a classification, but crisp statements capturing uniform structural constraints.
It is interesting that this can be done with respect to some natural parameter:
in 2013, Burness, Liebeck and Shalev \cite{BLS13} proved that every maximal subgroup of every finite simple group can be generated by at most $4$ elements.
The same result for almost simple groups (with $5$ in place of $4$)
was proven by Lucchini, Marion and Tracey in 2020 \cite{LMT20}.
A different type of constraint was noted by the second author in \cite{Sab26}.
He observed that, if $G$ is an almost simple group with a solvable maximal subgroup $H$, then the derived length of $H$ is at most $10$ \cite[Th.~1.4]{Sab26}.
This is sharp if $H$ is a solvable maximal subgroup of the sporadic Fischer group $\mathrm{Fi}_{23}$.

In this paper, we show that the latter result is just a glimpse of a much more general phenomenon, which, similarly to the theorems in \cite{BLS13,LMT20}, covers all maximal subgroups of all almost simple groups.
In fact, we prove that all maximal subgroups are nearly perfect,
in the sense that their derived series stabilize in a bounded number of steps.
We write $H^{(n)}$ for the $n$-th term of the derived series of $H=H^{(0)}$.

\begin{theorem} \label{thMain}
If $G$ is an almost simple group and $H$ is a maximal subgroup of $G$,
then $H^{(10)}$ is perfect.
In addition, $H^{(9)}$ is perfect unless $H$ is a solvable maximal subgroup of the Fischer group $\mathrm{Fi}_{23}$.
\end{theorem}

We recover \cite[Th.~1.4]{Sab26} when $H$ is solvable (and $H^{(10)}$ is trivial).
The proof of Theorem~\ref{thMain} relies on the Classification of the Finite Simple Groups, which allows us to split into cases where the structure of $H$ is well understood.
When $H$ is solvable, the work of Li and Zhang \cite{LZ11} yields immediate strong restrictions, but in the general case the analysis is much more delicate.
One of our tools is the observation that, for any normal subgroup $N$ of $H$, if both $N^{(a)}$ and $(H/N)^{(b)}$ are perfect then so is $H^{(a+b)}$ (Proposition~\ref{propGDL}).

It is natural to ask whether Theorem~\ref{thMain} can be extended to maximal subgroups of perfect groups. In general, the answer is no:

\begin{example} \label{exBad}
Let $X$ be a solvable group,
and let $G = (X \wr_5 \Alt(5))'$ be the derived subgroup of the natural wreath product.
It is not difficult to prove (see Lemma~\ref{lem:wreath'} below) that $G = Y \rtimes \Alt(5)$ is perfect,
where $Y$ is a subgroup of $X^5$ which contains a copy of $X'$.
If $K < \Alt(5)$ is a maximal subgroup,
then $H= Y \rtimes K$ is a solvable maximal subgroup of $G$ whose derived length can be arbitrarily large.
\end{example}

However, the maximal subgroup $H$ in Example~\ref{exBad} contains a large normal subgroup of $G$, which is the source of its large derived length.
In Section~\ref{sec4}, we show that this is the only possible obstruction:

\begin{theorem} \label{thPrimPerf}
If $G$ is a primitive perfect group with stabilizer $H$, then $H^{(10)}$ is perfect.
In addition, $H^{(9)}$ is perfect unless $H$ is a solvable maximal subgroup of the Fischer group $\mathrm{Fi}_{23}$.
\end{theorem} 

The proof of Theorem~\ref{thPrimPerf} is a combination of Theorem~\ref{thMain} and the O'Nan-Scott theorem.
Generalizing \cite{LS25},
it was shown in \cite[Th.~1.5]{Sab26} that in a primitive group with solvable stabilizer, there exist two points whose pointwise stabilizer has derived length bounded by an absolute constant.
The results in this paper suggest the following further generalization to all primitive groups:

\begin{conjecture} \label{conjTwoPoints}
In every finite primitive group,
there exist two points (possibly the same) whose pointwise stabilizer $H$ satisfies that $H^{(c)}$ is perfect,
where $c$ is an absolute constant.
\end{conjecture} 

We point out in Example~\ref{exBad2} that the constant in Conjecture~\ref{conjTwoPoints} is at least $7$.

\vspace{0.1cm} 
\section{Preliminaries} \label{sec2}

\subsection{A generalized derived length} 

Every group in this paper is finite.
For a finite group $G$,
let $G=G^{(0)}$ and
$$ G^{(n+1)} \> = \> [G^{(n)},G^{(n)}] $$
for any $n \geq 0$
(we also write $G'=G^{(1)}$).
Let $G^{(\infty)} = \cap_{n \geq 1} G^{(n)}$ be the perfect core (or solvable residual) of $G$.
Since $G$ is finite, $G^{(\infty)}$ is the last term of the derived series, or equivalently it is the unique largest perfect subgroup
(note that a group generated by perfect subgroups is perfect).

In order to deal with our main results, it is convenient to write $\dll(G)$ for the minimum $n$ such that $G^{(n)}=G^{(\infty)}$ (this is also the minimum $n$ such that $G^{(n)} = G^{(n+1)}$).
We may also write $\dl(G)$ if $G$ is solvable, i.e. if $G^{(\infty)} =1$.
Observe that $\dll(G) = \dl(G/G^{(\infty)})$ and that $\dll(G) =0$ if and only if $G$ is a (possibly trivial) perfect group.

\begin{lemma} \label{lemSub} 
    Let $G$ be a finite group and $H \leqslant G$.
    Then
    \begin{itemize}
        \item for each $n \geq 0$ we have $H^{(n)} \subseteq G^{(n)}$;
        in particular, $H^{(\infty)} \subseteq G^{(\infty)}$;
        \item if $G^{(\infty)} \subseteq H$, then $H^{(\infty)} = G^{(\infty)}$ and $\dll(H) \leq \dll(G)$.
    \end{itemize} 
\end{lemma}
\begin{proof}
    The first part is trivial.
    For the second part, we work in the solvable group $G/G^{(\infty)}$. Since subgroups of solvable groups are solvable with derived length not exceeding the original derived length, the result is clear.
\end{proof}

\begin{lemma} \label{lemSubNor} 
    Let $G$ be a finite group and $N \unlhd G$.
    Then
    \begin{itemize}
        \item for each $n \geq 0$ we have $(G/N)^{(n)} = NG^{(n)}/N$;
        in particular, $(G/N)^{(\infty)} = NG^{(\infty)}/N$;
        \item if $N \subseteq G^{(\infty)}$, then $\dll(G) = \dll(G/N)$.
    \end{itemize} 
\end{lemma}
\begin{proof}
    The first part is trivial.
    For the second part, observe that $G/G^{(\infty)}$ is isomorphic to $\frac{(G/N)}{(G/N)^{(\infty)}}$.
\end{proof}

A key difference between $\dll(G)$ and the genuine derived length, which is only defined for solvable groups, is that the first can increase on subgroups in general (just take a perfect group and a non-trivial solvable subgroup).
However, $\dll(G)$ behaves well with respect to group extensions.

\begin{proposition} \label{propGDL}
If $N \unlhd G$, then
$$ \dll(G/N) \> \leq \> \dll(G) \> \leq \> \dll(G/N) + \dll(N) . $$
\end{proposition} 
\begin{proof}
For each $n \geq 0$, it is easy to see that $(G/N)^{(n)}$ is perfect if $G^{(n)}$ is perfect, and so $\dll(G/N) \leq \dll(G)$.
For the right hand side, observe that this result is well-known (and trivial) for solvable groups. Therefore we have
\begin{align*}
    \dll(G) & =  \dl(G/G^{(\infty)}) \\
    & \leq \dl(G/G^{(\infty)}N) + \dl(G^{(\infty)}N/N) \\
    & = \dl(G/G^{(\infty)}N) + \dl(N/G^{(\infty)} \cap N) .
\end{align*}
Now $\dl(G/G^{(\infty)}N) \leq \dll(G/N)$ because the first group is a solvable quotient of the second.
For the same reason we get $\dl(N/G^{(\infty)} \cap N) \leq \dll(N)$,
and the proof follows.
\end{proof}

\begin{corollary} \label{corExt}
If $G=N_0 \rhd \cdots \rhd N_k = 1$ is a subnormal series where $N_i/N_{i+1}$ is solvable or perfect for each $i$,
then,
$$ \dll(G) \> \leq \> 
\sum_{N_i/N_{i+1} \mbox{ solvable}} \dl(N_i/N_{i+1}) . $$
\end{corollary}
\begin{proof}
Apply Proposition~\ref{propGDL} iteratively.
\end{proof}

Corollary~\ref{corExt} is particularly useful when some large abelian section appears in the subnormal series. Large $p$-sections are also nice, as explained in the following remark:

\begin{remark} \label{remP}
For a $p$-group $P$ we have $\dl(P) \leq 1$ if $|P| \leq p^2$,
$\dl(P) \leq 2$ if $|P| \leq p^5$,
$\dl(P) \leq 3$ if $|P| \leq p^{12}$,
$\dl(P) \leq 4$ if $|P| \leq p^{21}$,
and $\dl(P) \leq 5$ if $|P| \leq p^{39}$ \cite{Man00}.
\end{remark}

We report another useful lemma.

\begin{lemma} \label{lemSupp}
For every finite group $G$, there exists a solvable subgroup $S$ such that $G= G^{(\infty)} S$.
In particular, $\dll(G) \leq \dl(S)$.
\end{lemma}
\begin{proof}
We work by induction on $|G|$.
We can assume that $G^{(\infty)} \neq 1$ is non-solvable,
so $G^{(\infty)}$ is not solvable and not contained in the Frattini subgroup of $G$,
and there exists a maximal subgroup $M$ of $G$ such that $G=G^{(\infty)}M$.
By induction, there exists a solvable subgroup $S$ of $M$ such that $M=M^{(\infty)} S$.
But then $G = G^{(\infty)} M = G^{(\infty)} S$.

The second part follows because $G/G^{(\infty)}$ is isomorphic to a quotient of $S$.
\end{proof}

\begin{remark} \label{remRankBound}
Let $G \leqslant \GL_n(q)$.
By Lemma~\ref{lemSupp} we have $\dll(G) \leq \dl(S)$ for some solvable group $S \leqslant \GL_n(q)$.
Now, the derived length of $S$ is bounded by a function on $n$ \cite{New72}.
In particular, if we aim to prove $\dll(G) \leq 9$, looking at \cite[Pag.~70]{New72}, we can assume that $G$ does not have a faithful linear representation of dimension less than $9$.
\end{remark} 

A few times in the paper we deal with the derived subgroup of a particular wreath product. We isolate the following statement.

\begin{lemma} \label{lem:wreath'}
Let $G = X \wr_d S$ with $S \leqslant \Sym(d)$ for $d \geq 2$.
If $S$ is transitive and perfect, and $B=X^d$ is the base of the wreath product, then $G' = [B,S]S$, and $G'$ is perfect.
\end{lemma}
\begin{proof}
Write $B=X_1\times \cdots \times X_d$,
fix $i \neq j$ and $s\in S$ such that $i^s=j$, and $x,y\in X_i$.
Consider the element
$$ [xy,s][x,s]^{-1}[y,s]^{-1}\in [B,S] , $$
and compute
$$ [xy,s][x,s]^{-1}[y,s]^{-1} =
y^{-1}x^{-1}x^s y^s (x^s)^{-1}x (y^s)^{-1}y. $$
Since $x^s,y^s\in X_j$, and $X_i$ commutes with $X_j$, this simplifies to
$$ x^s y^s (x^s)^{-1}(y^s)^{-1} \in X_j'. $$
Moreover, as $x$ and $y$ vary in $X_i$, $[B,S]$ contains all commutators in $X_j$, so $X_j' \subseteq [B,S]$. Since $j$ was arbitrary, we obtain $B' \subseteq [B,S]$.

Now we have $G'= B' [B,S] S = [B,S]S$.
It follows that $[B,S] = [B,S'] \subseteq G^{(2)}$ and also $S = S' \subseteq G^{(2)}$, so $G' = G^{(2)}$ as desired.
\end{proof}

\subsection{Finite simple groups} 

Let $T$ be a non-abelian finite simple group.
The {\itshape outer automorphism group} $\Out(T)$ is the quotient of $\Aut(T)$ by the normal subgroup of the inner automorphisms (which is isomorphic to $T$).
The following result is a consequence of the Classification of the Finite Simple Groups.

\begin{lemma} \label{lemOut}
If $T$ is a finite simple group, then $\Out(T)$ is solvable of derived length at most $3$.
More precisely:
\begin{itemize}
\item If $T$ is cyclic, alternating, or a sporadic simple group, then $\Out(T)$ is abelian;
\item $\dl(\Out(T)) =3$ if and only if $T$ is the orthogonal group $\mathrm{P}\Omega^+_8(q)$ for $q$ odd.
\end{itemize}
\end{lemma}
\begin{proof}
This is standard (for example, see \cite[Lem.~3.2]{Sab26}).
\end{proof}

A finite group $G$ is {\itshape almost simple} if there exists a non-abelian simple group $T$ such that $T \unlhd G \leqslant \Aut(T)$.
By Lemma~\ref{lemOut}, the almost simple groups are nearly perfect in the sense that $\dll(G) \leq 3$, and Theorem~\ref{thMain} says that this property (with $10$ in place of $3$) is inherited by the maximal subgroups.

We now give some preliminaries for the classical groups of Lie type.
We follow the notation used in \cite{KL90}, which we summarize below. Let $V$ be a vector space over a finite field $\mathbb{F}$, and let $\kappa$ be either the zero form, or a non-degenerate unitary, symplectic, or quadratic form on $V$. We say that $\kappa$ is of type $\boldL, \boldU, \boldS$ or $\boldO$, respectively. We may refer to type $\boldU$ as $\boldL^-$. Additionally, type $\boldO$ splits into three further types: $\boldO^+$ and $\boldO^-$, when $\dim(V)$ is even; and $\boldO^\circ$, when $\dim(V)$ is odd. We will refer to $\boldO^\circ$ simply as $\boldO$.
As in \cite[Ch.~2]{KL90}, we define:
\begin{itemize}
    \item $S = S(V, \mathbb{F}, \kappa)$ the group of special $\kappa$-isometries of $V$;
    \item $I = I(V, \mathbb{F}, \kappa)$ the group of $\kappa$-isometries of $V$;
    \item $\Delta = \Delta(V, \mathbb{F}, \kappa)$ the group of $\kappa$-similarities of $V$;
    \item $\Gamma = \Gamma(V, \mathbb{F}, \kappa)$ the group of $\kappa$-semisimilarities of $V$;
    \item If $\kappa$ is of type $\boldO$, let $\Omega = \Omega(V, \mathbb{F}, \kappa)$ be an index $2$ subgroup of $S$, as defined in \cite[Sec.~2.5]{KL90}. Otherwise, let $\Omega = S$;
    \item If $\kappa$ is of type $\boldL$ with $\dim(V) \geq 3$, let $\Sigma = \Sigma(V, \mathbb{F}, \kappa) = \langle \Gamma, \iota \rangle$, where $\iota$ is the inverse transpose map on $\Gamma$. Otherwise, let $\Sigma = \Gamma$. 
\end{itemize}
We have the chain of groups
 $$    \Omega \> \leqslant \> S \> \leqslant \> I \> \leqslant \> \Delta \> \leqslant \>  \Gamma \> \leqslant \> \Sigma, $$
with $\Gamma$ acting on the vector space $V$.
For example, if $V=\mathbb{F}_q^n$ and $\kappa$ is the zero form, then
$\Omega = S = \SL_n(q)$, $I = \Delta = \GL_n(q)$, $\Gamma = \GammaL_n(q)$ and $\Sigma = \langle \Gamma, \iota \rangle$.

For any subgroup $K$ of $\Sigma$, we write $\overline K$ to denote the image of $K$ under the quotient by the scalars of $\Sigma$.
Similarly, for any subgroup $K$ of $\overline \Sigma$, we use $K^\wedge$ to denote the preimage of $K$ in $\Sigma$.

\begin{remark} \label{remark:exceptionalgraphs}
    If $\overline \Omega$ is a simple group, then the group $\overline \Sigma$ is always contained in $\Aut(\overline \Omega)$. In almost all cases $\overline \Sigma = \Aut(\overline \Omega)$, unless $\Omega$ is equal to $\Sp_4(q)$ with $q$ even, or $\Omega_8^+(q)$ for any $q$. Then, $\Aut(\overline \Omega)$ is equal to $\langle \overline \Sigma, \overline \gamma \rangle$, for some graph automorphism $\gamma$.
\end{remark}

\begin{lemma}\label{lemma:dlclassicalgroup}
    If $(V, \mathbb F, \kappa)$ is a classical geometry, then $ \Sigma / \Omega$ and $\overline \Sigma / \overline \Omega$ are solvable.
    If $\overline \Omega$ is simple, then $\dll(\overline \Sigma) \leq 2$. Otherwise, we have
    \[
        \dll( \overline \Sigma ) \> \leq \>
        \begin{cases*}
            3 \qquad & if $\kappa$ is of type $\boldL$ or $\boldS$,\\
            4 & if $\kappa$ is of type $\boldO$,\\
            5 & if $\kappa$ is of type $\boldU$.
        \end{cases*}
    \]
\end{lemma}
\begin{proof}
    In \cite[Ch.~2]{KL90} the groups $ \Sigma /  \Omega$ are classified, and it is easy to see that they are solvable with $\dl(\overline \Sigma/\overline \Omega) \leq 2$. This proves the result in the case where $\overline \Omega$ is simple.
    If $\overline \Omega$ is not simple, then the possibilities for $\overline \Omega$ are indicated in \cite[Prop.~2.9.2]{KL90}. In the case that $q \leq 3$, we check that $\dll(\overline \Sigma)$ satisfies the given bound computationally, see Table~\ref{table:dllnonsimpleomega}. If instead $q > 3$, then $\Omega$ must be one of $\SL_1(q)$, $\Omega_2^{\pm}(q)$ or $\Omega_4^+(q)$. In the first two cases $\Omega$ is cyclic, so $\dl(\overline \Omega) \leq 1$ and hence $\dl(\overline \Sigma) \leq 3$ by Proposition~\ref{propGDL}. Finally, if $\Omega$ is equal to $\Omega_4^+(q)$, then $\Omega$ is isomorphic to $\SL_2(q) \circ \SL_2(q)$, which is perfect since $q > 3$. Thus, $\overline \Omega = \overline \Sigma^{(\infty)}$ and $\dll(\overline \Sigma) = \dl(\overline \Sigma/\overline \Omega) \leq 2$.

    \begin{table}[]
        \centering
        \begin{tabular}{llrr}
            Case & $(n, q)$ & $\dll(\overline I)$ & $\dl(\overline \Sigma / \overline I)$ \\ \toprule
            $\boldL$ & $(2,2)$ & 2 & 0 \\
            & $(2,3)$ & 3 & 0 \\
            $\boldU$ & $(3,2)$ & 4 & 1 \\
            $\boldS$ & $(4,2)$ & 1 & 1 \\
            $\boldO$ & $(2,q,\pm)$ & 2 & 1 \\
            & $(4,2,+)$ & 3 & 0 \\
            & $(4,3,+)$ & 3 & 1 \\ \bottomrule
        \end{tabular}
        \caption{The derived length of $\overline I$ and $\overline \Sigma / \overline I$ in the case that $\overline \Omega$ is not simple and $q \leq 3$.}
        \label{table:dllnonsimpleomega}
    \end{table}
\end{proof}

\vspace{0.1cm}
\section{Proof of Theorem~\ref{thMain}} \label{sec3}

Let $G$ be an almost simple group with socle $T$,
and let $H$ be a maximal subgroup of $G$.
We show that $\dll(H) \leq 10$.
By Lemma~\ref{lemOut} we have that $H/T \cap H$ is solvable (as a subgroup of $G/T$), and so $H^{(\infty)} \subseteq T \cap H$. It follows by Lemma~\ref{lemSub} and Proposition~\ref{propGDL} that
$$ \dll(T \cap H) \> \leq \>
\dll(H)
\> \leq \> \dll(T \cap H) + \dl(G/T) . $$
In most cases, it will be sufficient to use the weaker bound
$$ \dll(H) \> \leq \> \dll(T \cap H) + \dl( \Out(T) ) . $$
Note that, in general, $T \cap H$ is not a maximal subgroup of $T$.

Using the Classification of the Finite Simple Groups,
we distinguish various cases.

\subsection{$T$ is alternating} 

We address this case with the following proposition.
To simplify the argument, we use the main theorem of \cite{LPS87}.

\begin{proposition}
Let $G=\Sym(n)$ or $G=\Alt(n)$ and let $H$ be a maximal subgroup of $G$.
Then $\dll(H) \leq 6$, and $\dll(H) \leq 4$ with finitely many exceptions.
\end{proposition}
\begin{proof}
If $H$ is almost simple, then $\dll(H) \leq 3$ by Lemma~\ref{lemOut}.
Otherwise, by \cite{LPS87}, $H = X \cap G$ where $X$ is a maximal subgroup of $\Sym(n)$ which is not almost simple. If $X \neq H$, then $|X:H| =2$ and $X' \subseteq H$.
In general, we always have $\dll(H) \leq \dll(X)$ by Lemma~\ref{lemSub}.

If $X$ is intransitive, then $X = \Sym(k) \times \Sym(n-k)$ for some $k \geq 1$. Thus $\dll(X) = \max \{ \dll(\Sym(k),\dll(\Sym(n-k))) \} \leq 3$.

If $X$ is transitive but not primitive, then $n=md$, $X= \Sym(m) \wr_d \Sym(d)$ and by Proposition~\ref{propGDL} we have $\dll(X) \leq \dll(\Sym(m)) + \dll(\Sym(d)) \leq 6$
(equality is attained for $m=d=4$).
Moreover $\dll(H) \leq 4$ if $md > 16$
(equality is attained for $d=4$ and $m \geq 5$).

If $X$ is primitive,
then by the O'Nan-Scott theorem \cite[Sec.~4]{DM96} we have the following possibilities:
\begin{enumerate}
    \item[$\bullet$] \textbf{Affine.}
    In this case $n=p^d$ and $X = \AGL_d(p)$.
    Therefore $\dll(X) \leq 5$, and $\dll(X) \leq 4$ if $p^d >3^2$;

    \item[$\bullet$] \textbf{Product action.}
    In this case $n=m^d$ and $X = \Sym(m) \wr_d \Sym(d)$.
    Therefore $\dll(X) \leq 6$ and $\dll(X) \leq 4$ if $m^d > 4^4$;
    
    \item[$\bullet$] \textbf{Diagonal type.}
    In this case $n=|S|^{m-1}$, where $S$ is a non-abelian simple group, and $X = S^m.(\Out(S) \times \Sym(m))$. Hence, $\dll(X) \leq 3$.
\end{enumerate}

The proof is complete.
\end{proof}

In the exceptional case where $T=\Alt(6)$, it can be checked easily that $\dll(H) \leq 3$.

\subsection{$T$ is a classical group} 

Let $T$ be a classical simple group over $\mathbb{F} = \mathbb F_{q^u}$ where $q=p^r$ and $u$ is either $2$ if $T$ is of type $\boldU$, or $1$ otherwise.
Then there exists a classical geometry $(V, \mathbb{F}, \kappa)$ such that $T = \overline \Omega(V, \mathbb{F}, \kappa)$, and $\overline \Omega \leqslant G \leqslant \Aut(\overline \Omega)$. We assume first that $G \subseteq \overline \Sigma$, and treat the exceptional cases at the end of this section.
We split into different types of maximal subgroups $H<G$ by using Aschbacher's Theorem (see \cite[Ch.~3]{KL90}).
 We assume that $H$ is not solvable, as this case is already treated in \cite[Th.~1.4]{Sab26}.
Finally, we delay the case where $H$ is parabolic to Section~\ref{sec3.4}, where we are able to consider all groups of Lie type simultaneously.

 \begin{enumerate}
    \item $\mathscr C_1$: We only treat the non-parabolic cases.
    Suppose $H$ is of type $\mathscr C_1$. We assume first that $G$ is of type $\boldL$, and $H$ is of type $\GL_m(q) \oplus \GL_{n-m}(q)$, so $H$ stabilizes a vector space decomposition $V=V_1 \oplus V_2$ with $\dim(V_1) = m$ and $\dim(V_2) = n-m$. Recall that $H^\wedge$ denotes the preimage of the subgroup $H$ of $\overline \Sigma$ in $\Sigma$. Then, we have that
    \[
        \SL_m(q) \times \SL_{n-m}(q) 
        \leqslant H^\wedge \cap \Gamma \leqslant
        \GammaL_m(q) \times \GammaL_{n-m}(q).
    \]
    Since $\GammaL_i(q)/\SL_i(q)$ is solvable for all $i$, this implies that $(\GammaL_m(q) \times \GammaL_{n-m}(q))^{(\infty)}$ is contained in $H^\wedge \cap \Gamma$. As such, by Lemma~\ref{lemSub},
    \begin{align*}
        \dll(H \cap \overline \Gamma) & \leq \dll(H^\wedge \cap \Gamma)\\
        & \leq \dll(\GammaL_m(q) \times \GammaL_{n-m}(q))\\
        & \leq \max(\dll(\GammaL_m(q)), \dll(\GammaL_{n-m}(q))) \\
        & \leq 4,
    \end{align*}
    where the final inequality is by Lemma~\ref{lemma:dlclassicalgroup}. As such, by Proposition~\ref{propGDL},
    \[
        \dll(H) \leq \dll(H \cap \overline \Gamma) + \dl(\overline \Sigma/\overline \Gamma) \leq 5.
    \]
    
    Assume that $H$ is of type $\GU_m(q) \perp \GU_{n-m}(q)$, $\Sp_m(q) \perp \Sp_{n-m}(q)$ or $\O_m^{\epsilon_1}(q) \perp \O_{n-m}^{\epsilon_2}(q)$. In these cases $G$ is not of type $\boldL$, and as such $\Sigma = \Gamma$. There exist non-degenerate vector subspaces $V_1, V_2$ of $V$ with $V = V_1 \perp V_2$, and $H$ equal to $N_G(V_1)$. Define forms $\kappa_i$ on each $V_i$ by restricting $\kappa$ to $V_i$. Then,
    \[
        \Omega(V_1, \mathbb F, \kappa_1) \times \Omega(V_2, \mathbb F, \kappa_2)
        \leqslant H^{\wedge} \leqslant
        \Gamma(V_1, \mathbb F, \kappa_1) \times \Gamma(V_2, \mathbb F, \kappa_2).
    \]
    The groups $\Gamma(V_i, \mathbb F, \kappa_i)/\Omega(V_i, \mathbb F, \kappa_i)$ are solvable, and hence
    \[
        \dll(H) \leq \dll(H^\wedge) \leq \max_i \left( \dll(\Gamma(V_i, \mathbb F, \kappa_i))\right) \leq 6.
    \]
    
    Finally, if $H$ is of type $\Sp_{n-2}(q)$ then $H$ is a classical group of type $\boldS$, and so $\dll(H) \leq 3$.
    
    \item $\mathscr C_2$: Suppose $H$ is of type $\GL_m(q) \wr_t \Sym(t)$. Then $H$ stabilizes a vector space decomposition $V=V_1 \oplus \cdots \oplus V_t$ with $\dim(V_i) = m$ for all $i$. Thus,
    \[
        \SL_m(q) \wr_t \Alt(t) \leqslant H^{\wedge} \cap \Gamma \leqslant
        \GammaL_m(q) \wr_t \Sym(t).
    \]
    Let $J$ be the projection of $H$ to $\Sym(t)$, which must be either $\Alt(t)$ or $\Sym(t)$. Hence, by Proposition~\ref{propGDL},
    \[
        \dll(H^\wedge \cap \Gamma) \leq \dll(\GammaL_m(q)^t \cap H) + \dll(J).
    \]
    However, $\GammaL_m(q)^t \cap H$ contains $\SL_m(q)^t$, and hence contains $(\GammaL_m(q)^t)^{(\infty)}$. Thus, by Lemma~\ref{lemSub},
    \begin{align*}
        \dll(H^\wedge \cap \Gamma) & \leq \dll(\GammaL_m(q)^t) + \dll(J) \\
        & \leq \dll(\GammaL_m(q)) + \dll(\Sym(t)).
    \end{align*}
    If $\PSL_m(q)$ is a simple group, then $\dll(\GammaL_m(q)) \leq 3$ by Lemma~\ref{lemma:dlclassicalgroup}. Since $\dll(\Sym(t)) \leq 3$ for all $t$, this implies that $\dll(H^{\wedge} \cap \Gamma) \leq 6$. On the other hand, if $\PSL_m(q)$ is not simple we have that $\dll(\GammaL_m(q)) \leq 4$. However, in this case we may assume that $\Alt(t)$ is simple, so that $H$ is non-solvable, and hence $\dll(\Sym(t)) = 1$. Hence we have that $\dll(H^\wedge \cap \Gamma) \leq 5$. Thus, by Proposition~\ref{propGDL},
    \[
        \dll(H) \leq \dll(H^{\wedge}) \leq \dll(H^{\wedge} \cap \Gamma) + \dl(\Sigma/\Gamma) \leq 7.
    \]

    Suppose next that $H$ is of type $\GU_m(q) \wr_t \Sym(t)$, $\Sp_m(q)\wr_t \Sym(t)$, or $\O_m^\epsilon(q) \wr_t \Sym(t)$. Note that $G$ is not of type $\boldL$ in this case. Then there exist some pairwise isometric, non-degenerate vector spaces $V_1, \dots, V_t$ of $V$ such that $V = V_1 \perp \cdots \perp V_t$, with $H$ equal to $N_G\{V_1, \dots, V_t\}$. Define forms $\kappa_i$ on each $V_i$ by restricting $\kappa$ to $V_i$, and let $\Omega_i = \Omega(V_i, \mathbb F, \kappa_i)$, and $\Gamma_i = \Gamma(V_i, \mathbb F, \kappa_i)$ for each $i$. Note that all of the groups $\Omega_i$ are isomorphic, as are all of the groups $\Gamma_i$. Then,
    \[
        \Omega_1 \wr_t \Alt(t) 
        \leqslant H^{\wedge} \leqslant
        \Gamma_1 \wr_t \Sym(t).
    \]
    Hence, similarly to the $\boldL$ case,
    \begin{align*}
        \dll(H^{\wedge}) & \leq \dll(\Gamma_1) + \dll(\Sym(t)) \\
        & \leq \begin{cases*}
            3 + 3 & if $\overline \Omega_1$ is simple\\
            6 + 1 & if $t \geq 5$.
        \end{cases*}
    \end{align*}
    So $\dll(H) \leq 7$, assuming $H$ is non-solvable.

    If $H$ is of type $\GL_{n/2}(q^u).2$, then $\SL_{n/2}(q^u) \leqslant H^\wedge \cap \Omega \leqslant \GL_{n/2}(q^u).2$, and so by Proposition~\ref{propGDL} and Lemma~\ref{lemma:dlclassicalgroup} we have that 
    \begin{align*}
        \dll(H) & \leq \dll(H^\wedge \cap \Omega) + \dl(\overline \Sigma/\overline \Omega) \\
        & \leq \dll(\GL_{n/2}(q^u)) + \dl(C_2) + \dl(\overline \Sigma/\overline \Omega)\\
        & \leq 4 + 1 + 2.
    \end{align*}

    If $H$ is of type $\O_{n/2}(q)^2$, then $H^\wedge$ is contained in $\GammaO_{n/2}(q)^2 : C_2$ and contains $\Omega_{n/2}(q)^2$, so as previously we have that $\dll(H^\wedge) \leq 6$.
    
    \item $\mathscr C_3$: For $H$ in this class there exists a field extension $\mathbb F_\sharp$ of $\mathbb F$ and a classical geometry $(V_\sharp, \mathbb{F}_\sharp, \kappa_\sharp)$ such that
    \[
        \Omega(V_\sharp, \mathbb F_\sharp, \kappa_\sharp) 
        \leqslant H^\wedge \cap \Gamma \leqslant
        \Gamma(V_\sharp, \mathbb{F}_\sharp, \kappa_\sharp).
    \]
    When $G$ is not of type $\boldL$, we have that $\Sigma = \Gamma$, and so $H^{\wedge} = H^\wedge \cap \Gamma$. Thus, by Lemmas \ref{lemSub} and \ref{lemma:dlclassicalgroup},
    \[
        \dll(H) \leq \dll(\Gamma(V_\sharp, \mathbb{F}_\sharp, \kappa_\sharp)) \leq 6.
    \]
    On the other hand, if $G$ is of type $\boldL$, then $\Sigma = \Gamma :2$. Note that, by \cite[Sec.~4.3]{KL90}, the geometry $(V_\sharp, \mathbb F_\sharp, \kappa_\sharp)$ is also of type $\boldL$ in this case. Hence, by Lemma~\ref{lemma:dlclassicalgroup} and Proposition~\ref{propGDL},
    \[
        \dll(H) \leq \dll(\Gamma(V_\sharp, \mathbb{F}_\sharp, \kappa_\sharp)) + 1 \leq 5.
    \]
    
    \item $\mathscr C_4$: Assume $H$ is of type $\GL_{n_1}(q) \otimes \GL_{n_2}(q)$, where $n_1 n_2 = n$. Then, $H$ stabilizes a tensor product decomposition $V=V_1 \otimes V_2$, with $\dim(V_i) = n_i$. We have
    \[
        \PSL_{n_1}(q) \times \PSL_{n_2}(q) 
        \leqslant H \cap \overline \Gamma \leqslant
        \PGammaL_{n_1}(q) \times \PGammaL_{n_2}(q).
    \]
    Thus, as in the $\mathscr{C}_1$ case, this implies that $\dll(H \cap \overline \Gamma) \leq 3$ and so $\dll(H) \leq 4$.

    If $H$ is instead of type $\GU_{n_1}(q) \otimes \GU_{n_2}(q)$, $\Sp_{n_1}(q) \otimes \O_{n_2}^\epsilon (q)$, $\Sp_{n_1}(q) \otimes \Sp_{n_2} (q)$ or $\O_{n_1}^{\epsilon_1}(q) \otimes \O_{n_2}^{\epsilon_2} (q)$, then $H$ stabilizes a vector space decomposition $V=V_1 \otimes V_2$ for some vector spaces $V_1, V_2$ with forms $\kappa_i$ on each $V_i$. Thus,
    \[
        \overline \Omega(V_1, \mathbb F, \kappa_1) \times \overline \Omega(V_2, \mathbb F, \kappa_2) 
        \leqslant H \leqslant
        \overline \Gamma(V_1, \mathbb F, \kappa_1) \times \overline \Gamma(V_2, \mathbb F, \kappa_2).
    \]
    We hence obtain that $\dll(H) \leq 6$, as before.
    
    \item $\mathscr C_5$: For $H$ in this class, there exists a subfield $\mathbb F_\sharp$ of $\mathbb F$ and a classical geometry $(V_\sharp, \mathbb{F}_\sharp, \kappa_\sharp)$ as described in \cite[Sec.~4.5]{KL90}, such that
    \[
        \overline \Omega(V_\sharp, \mathbb{F}_\sharp, \kappa_\sharp) 
        \leqslant H \cap \overline \Delta \leqslant
        \overline \Delta(V_\sharp, \mathbb{F}_\sharp, \kappa_\sharp)
    \]
    So $\dll(H \cap \overline \Delta) \leq 5$ and hence $\dll(H) \leq 6$.
    
    \item $\mathscr C_6$: Recall we exclude any types for which $H$ is always solvable. If $H$ is of type $(4 \circ 2 ^{1+2m}). \Sp_{2m}(2)$ with $m > 1$, then $H \leqslant (2^{2m} . \Sp_{2m}(2)) \times 2$. If $m=2$, then $\Sp_4(2)' = \Alt(6)$, and $\Alt(6)$ acts irreducibly on the natural module of $\Sp_4(2)$. Hence, $H^{(\infty)}$ contains $2^4.\Alt(6)$, so $\dll(H) \leq 1$. Otherwise, if $m \geq 3$, then $H^{(\infty)}$ contains $2^{2m}.\Sp_{2m}(2)$, so $\dll(H) \leq 1$. 
    
    Similarly, if $H$ is of type $r^{1+2m}.\Sp_{2m}(r)$, with $r$ an odd prime and $m > 1$, or if it is of type $2_\pm^{1+2m}.\O_{2m}^\pm(2)$ with $m > 2$, then $\dll(H) \leq 1$. 
    
    \item $\mathscr C_7$: Assume first that $H$ is of type $\GL_m(q) \wr_t \Sym(t)$. This implies that $H$ stabilizes a tensor product decomposition $V = V_1 \otimes \cdots \otimes V_t$ for some vector spaces $V_i$ with $\dim(V_i) = m$. Note that $\PSL_m(q)$ must be simple in this case. Then, 
    \[
        \PSL_m(q) \wr_t \Alt(t) 
        \leqslant H \cap \overline \Gamma \leqslant
        \PGammaL_m(q) \wr_t \Sym(t).
    \]
    Thus, as in the $\mathscr C_2$ case, we can obtain that $\dll(H) \leq 6$. Similarly, for the other types of $H$ in this class, $H$ stabilizes a decomposition $V=V_1 \otimes \cdots \otimes V_t$. For each $i$ we can define a form $\kappa_i$ on $V_i$ as in \cite[Section 4.7]{KL90}, with $V_1, \dots, V_t$ pairwise similar. Then,
    \[
        \overline \Omega(V_1, \mathbb F, \kappa_1) \wr_t \Alt(t) 
        \leqslant H \leqslant
        \overline \Gamma(V_1, \mathbb F, \kappa_1) \wr_t \Sym(t).
    \]
    Note that $\overline \Omega(V_1, \mathbb F, \kappa_1)$ must be non-solvable, by definition. Hence $\dll(H) \leq 5$.
    
    \item $\mathscr C_8$: Let $(V_\sharp, \mathbb{F}_\sharp, \kappa_\sharp)$ be a classical geometry as described in \cite[Sec.~4.8]{KL90}, and let $\Gamma_\sharp = \Gamma(V_\sharp, \mathbb{F}_\sharp, \kappa_\sharp)$. Then, $H \leqslant \overline \Gamma_\sharp . 2$, and hence $\dll(H) \leq 6$.
    
    \item $\mathscr S$: In this case $\dll(H) \leq 3$ by Lemma~\ref{lemOut}.
\end{enumerate}

Finally, we consider the case where $G$ is not a subgroup of $\overline \Sigma$. As mentioned in Remark~\ref{remark:exceptionalgraphs}, we must have that $T =  \overline \Omega$ is equal to $\PSp_4(q)$ or $\POmega_8^+(q)$. Additionally, the group $\Aut(\overline \Omega)$ is generated by $\overline \Sigma$ and a graph automorphism of order $2$ or $3$ respectively. 

If $H \cap \overline \Sigma$ is solvable, then so is $H$ and we are done. If $H \cap \overline \Sigma$ is one of the non-solvable groups considered above, then we have proven that $\dll(H \cap \overline \Sigma) \leq 7$, and hence $\dll(H) \leq 8$. Thus, we can assume that $H \cap \overline \Sigma$ is not in the classes $\mathscr{C}_1$ to $\mathscr{C}_8$, nor $\mathscr{S}$.

We use \cite{BHR13} to study the remaining cases.
If $T = \PSp_4(q)$, then $H$ is solvable.
Hence we assume $T = \POmega_8^+(q)$. Let $d=(q-1,2)$.
We have the following possibilities for $H^{\wedge} \cap \Omega$, assuming it is non-solvable and not contained in $\mathscr C_1$-$\mathscr C_8$ nor $\mathscr S$: $d \times G_2(q)$; $(\Omega_2^\pm(q) \times \frac{1}{d} \GL_3(q)).[2d]$; and $2^4.2^6.\L_3(2)$. We treat each case individually.

If $H^\wedge \cap \Omega = d \times G_2(q)$, we can assume that $G_2(q)$ is not solvable, and hence $\dll(H^{\wedge} \cap \Omega) =1$. Thus, by Proposition~\ref{propGDL} and Lemma~\ref{lemOut}, $\dll(H) \leq 4$.
If $H^\wedge \cap \Omega = (\Omega_2^\pm(q) \times \frac{1}{d} \GL_3(q)).[2d]$, then $\dll(H \cap \overline \Omega) \leq 3$ and so $\dll(H) \leq 6$. If $H^\wedge \cap \Omega = 2^4.2^6.\L_3(2)$, then $\dll(H \cap \overline\Omega) \leq 2 + \dll(\L_3(2)) = 2$, and so $\dll(H) \leq 5$.\\

To conclude this part, we observe that there are infinitely many almost simple classical groups $G$ and maximal subgroups $H$ with $\dll(H)=6$:

\begin{example}
    Let $G = \Aut(\PSU_n(2))$ with $n \geq 7$, $3 \nmid n$,
    and let $H$ be a maximal subgroup of $G$ of type $\GU_3(2) \perp \GU_{n-3}(2)$.
    In particular, there exists a non-degenerate subspace $V_1$ of dimension $3$ such that $H = N_G(V_1)$.
    Letting $V_2 = V_1^\perp$,
    there exist unitary forms $\kappa_i$ on $V_i$,
    given simply by restricting $\kappa$ to $V_i$.
    Let $I_i = I(V_i, \mathbb F, \kappa_i)$ and $\Gamma_i = \Gamma(V_i, \mathbb F, \kappa_i)$. Note that $\Gamma_1 \cong \GU_3(2):C_2$, and $\dl(\Gamma_1) = 6$.
    Then
    $I_1 \times I_2 \leqslant H^\wedge \leqslant \Gamma_1 \times \Gamma_2$,
    as a subdirect product,
    where the scalars of $G$ in $H$ are $\diag(Z(I_1) \times Z(I_2))$.
    Since $3 \nmid n$, the group $Z(\SU_{n-3}(2))$ is trivial, and thus $\GU_{n-3}(2) = \SU_{n-3}(2) : Z(I_2)$. This implies that both $H^\wedge$ and $H$ have perfect cores isomorphic to $\PSU_{n-3}(2)$. Moreover, $H^{(5)}$ is isomorphic to
    \begin{align*}
        \left( \frac{H^\wedge}{\diag(Z(I_1) \times Z(I_2))} \right)^{(5)} & = \frac{\diag(Z(I_1) \times Z(I_2)) (H^\wedge)^{(5)}}{\diag(Z(I_1) \times Z(I_2))}\\
        & = \frac{Z(I_1) \times \GU_{n-3}(2)}{\diag(Z(I_1) \times Z(I_2))}.
    \end{align*}
    As such, $H^{(5)}$ is not isomorphic to $H^{(\infty)}$, and so $\dll(H) \geq 6$. Since $\dll(\Gamma_i) \leq 6$, we conclude that $\dll(H) = 6$.
\end{example} 

\subsection{$T$ is exceptional} 

By Remark \ref{remRankBound},
we can assume that $T$ is one among
$$ F_4(q) , \> ^2F_4(q) , \>
E_6(q) , \> ^2E_6(q) , \>
E_7(q) , \> \mbox{ or } E_8(q) . $$

Let $\widetilde G$ be the algebraic group corresponding to $T$, and let $\sigma$ be a Steinberg endomorphism such that $\widetilde G_\sigma' = T$. Note that $\widetilde G_\sigma$ is equal to $\operatorname{Inndiag}(T)$, the subgroup of $\Aut(T)$ generated by $T$ and the diagonal automorphisms of $T$.

We will use the following result from \cite{LS90}.

\begin{theorem}{\cite[Th.~2]{LS90}} \label{theorem:maxclassification}
    Let $G$ be an almost simple group with socle $T$ an exceptional group of Lie type, and let $H<G$ be a non-parabolic maximal subgroup.
    Then one of the following holds:
    \begin{itemize}
        \item[(i)] $H = N_G(D_\sigma)$, where $D$ is a reductive subgroup of $\widetilde G$ of maximal rank, the possibilities are given in \cite[Tab.~5.1, 5.2]{LSS90};
        \item[(ii)] $H$ is the centralizer of a graph, field, or graph-field automorphism of prime order;
        \item[(iii)] The generalized Fitting subgroup $F^*(H)$ is as in \cite[Tab.~III]{LS90};
        \item[(iv)] $H = N_G(E)$ where $E$ is an elementary abelian group given in \cite[Th.~1(II)]{CLSS92};
        \item[(v)] $T \cap H$ is one of $(2^2 \times D_4(q)) . \Sym(3)$, $A_1(q) \times \Sym(5)$, or $(\Alt(5) \times \Alt(6)).2^2$;
        \item[(vi)] $H$ is almost simple.
    \end{itemize}
\end{theorem}

We start analyzing the various possibilities.
If $H$ is of type (i), and $D$ is not a torus, the possibilities for $N_{\widetilde G_\sigma}(D_\sigma)$ are given in \cite[Tab.~5.1]{LSS90}. By Proposition~\ref{propGDL} we have $\dll(H) \leq \dll(H \cap \widetilde G _\sigma) + \dll(H/ H \cap \widetilde G_\sigma)$.
Now $H \cap \widetilde G _\sigma = N_{\widetilde G_\sigma}(D_\sigma) \cap G$, so $N_{\widetilde G_\sigma}(D_\sigma)^{(\infty)} \subseteq H \cap \widetilde G_\sigma$ and $\dll(H \cap \widetilde G_\sigma) \leq \dll(N_{\widetilde G_\sigma}(D_\sigma))$ by Lemma~\ref{lemSub}.
Moreover, $H/ H \cap \widetilde G_\sigma$ is isomorphic to a subgroup of the solvable group $G \widetilde G_\sigma/ \widetilde G_\sigma$.
Therefore
\[
    \dll(H) \> \leq \>
    \dll(N_{\widetilde G_\sigma}(D_\sigma)) + \dl(G/ G \cap \widetilde G_\sigma).
\]
Using this inequality, we can prove that $\dll(H) \leq 9$ for all groups in \cite[Tab.~5.1]{LSS90}.
We remark on one of the more complex cases: suppose $T=E_8(q)$ and $N_{\widetilde G_\sigma}(D_\sigma)= d^4.(\L_2(q))^8.d^4.\AGL_3(2)$. If $q$ is $2$ or $3$, we can directly construct this group in Magma \cite{magma}, and hence show that $\dll(N_{\widetilde G_\sigma}(D_\sigma)) \leq 1$. When $q \geq 4$ instead, we can show that $d^4:\AGL_3(2)$ is perfect and hence $N_{\widetilde G_\sigma}(D_\sigma)$ is also perfect. The other cases are easy, with the largest bounds on $\dll(H)$ usually occurring when $q$ is $2$ or $3$.

If $H$ is of type (i) and $D$ is a maximal torus, we similarly consider every case in \cite[Tab.~5.2]{LSS90}. In each case we can show that $\dll(N_{\widetilde G_\sigma}(D_\sigma)/D_\sigma) \leq 4$, so by Corollary \ref{corExt} we have that $\dll(H) \leq \dll(D_\sigma) + 4 + \dll(G/(\widetilde G_\sigma \cap G)) \leq 7$.

If $H$ is of type (ii), then there exists $q_0 = q^{1/r}$, with $S$ equal to one of $F_4(q_0)$, ${}^2F_4(q_0)$, $E_6(q_0)$, ${}^2E_6(q_0)$, $E_7(q_0)$ or $E_8(q_0)$, such that $S \leqslant T \cap H \leqslant \Aut(S)$. Since $S$ is always simple, we have $\dll(T \cap H) \leq 2$, and thus $\dll(H) \leq 4$. 

Next, if $H$ is of type (iii), then $F^*(H)$ is the direct product of two or three non-abelian simple groups which are given in \cite[Tab.~III]{LS90}.
Moreover, $H$ is a subgroup of $\Aut(F^*(H))$ containing $F^*(H)$,
and it is easy to see that $\dll(H) \leq 5$.
For instance, suppose that $T = E_8(q)$, and $F^*(H) = \L_2(q) \times G_2(q)^2$ where $q > 3$. Then $H$ is contained in $\Aut(F^*(H)) = \Aut(\L_2(q)) \times (\Aut(G_2(q)) \wr C_2)$, and contains $\Aut(F^*(H))^{(\infty)} = F^*(H)$. As such, by Lemma~\ref{lemSub},
\[
    \dll(H) \leq \dll(\Aut(F^*(H))) \leq \max \left( \dl(\Out(\L_2(q))), \dl(\Out(G_2(q)) \wr C_2) \right) .
\]
Now $\dl(\Out(\L_2(q))) \leq 2$ by Lemma~\ref{lemOut}, and 
\[
    \dl(\Out(G_2(q)) \wr C_2) \leq \dl(\Out(G_2(q)) + 1 \leq 3.
\]
Hence, $\dll(H) \leq 5$.

If $H$ is of type (iv) then $H \cap \widetilde G_\sigma = N_{\widetilde G_\sigma}(E)$ for some elementary abelian group $E$, and $N_{\widetilde G_\sigma}(E)$ is described in \cite[Tab.~1]{CLSS92}. As such, we can find that
\begin{align*}
    \dll(H) & \leq
    \dll(C_{\widetilde G _\sigma}(E)) + \dll(N_{\widetilde G_\sigma}(E)/C_{\widetilde G_\sigma}(E)) + \dl(G/(G \cap \widetilde G_\sigma)) \\
    & \leq 2 + 2 + 2 = 6.
\end{align*}

If $H$ is of type (v), then $\dll(T \cap H) \leq 4$ and so $\dll(H) \leq 6$.
Finally, if $H$ is of type (vi) then $\dll(H) \leq 3$ by Lemma~\ref{lemOut}.

\subsection{$H$ is a parabolic subgroup} \label{sec3.4}

Let $G$ be a group of Lie type with socle $T$ and let $H<G$ be a parabolic maximal subgroup.
As in the previous section, let $\widetilde G$ be the algebraic simple group corresponding to $T$. Let $\Phi$ be a root system of $\widetilde G$, and $\Pi$ its simple system. For any subset $J \subset \Pi$, let $\Phi_J$ be the root subsystem of $\Phi$ spanned by $J$, and let $\Phi_J^+, \Phi^+$ denote their respective positive systems. For a root $\alpha \in \Phi^+ \setminus \Phi_J^+$, we can write $\alpha = \sum_{\beta \in J} b_\beta \beta + \sum_{\beta \in \Pi \setminus J} c_\beta \beta$, and we say
\begin{align*}
    \operatorname{height}(\alpha) & = \sum b_\beta + \sum c_\beta , \\
    \operatorname{level}(\alpha) & = \sum c_\beta .
\end{align*} 
For instance, if $G$ is of type $E_6$, then the root of $\Phi$ of largest height is
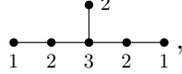
\[
    \dynkin[edge length=.5cm, labels={1, 2, 2, 3, 2, 1}, label directions={below,right,below,below,below,below}] E{6},
\]
which has height $11$. If $J$ is, for example, $\Pi$ with the middle node removed, then the level of this root is $3$, and this is the largest possible level of a root in this case.

For any root $\alpha \in \Phi$, let $X_\alpha$ be the root subgroup of $T$ corresponding to $\alpha$.
Since $H$ is a parabolic subgroup of $G$, there exists a subset $J$ of $\Pi$ such that $T \cap H$ is conjugate to $P_J$, where $P_J = U \rtimes L$ with 
\[
    U \> = \> \prod_{\alpha \in \Phi^+ \setminus \Phi_J^+} X_\alpha
\]
a $p$-group, and
\[
    L \> = \> \langle \prod_{\alpha \in J} X_\alpha, K \rangle 
\]
where $K$ is a Cartan subgroup of $T$.

\begin{remark}\label{remark:sizeofJ}
    If $J_1 \subset J_2 \subset \Pi$, then $P_{J_1} \subset P_{J_2}$. Thus, since $H$ is a maximal parabolic subgroup, $J$ must be a maximal subset of $\Pi$ such that $J^\gamma = J$ for any graph automorphism $\gamma \in G$. Hence, if $G$ is untwisted and contains no graph automorphisms, $|\Pi \setminus J| = 1$. If $G$ is twisted or contains a graph automorphism, then $\Pi \setminus J$ is equal to the orbit of a point under a graph automorphism, and so $|\Pi \setminus J|$ is either 2 or 3.
\end{remark}

By Proposition~\ref{propGDL} we have
$$
 \dll(T \cap H) \> \leq \> \dl(U) + \dll(L) ,
$$
and we will bound $\dl(U)$ and $\dll(L)$ separately.

\begin{lemma} \label{lemUnip}
We always have $\dl(U) \leq 3$.
Moreover, $\dl(U) \leq 2$ if $G$ is classical.
\end{lemma}
\begin{proof}
Let $U(i)$ denote the subgroup of $U$ given by $\prod X_{\alpha}$ where the product is over all $\alpha \in \Phi^+ \setminus \Phi_J^+$ with level at least $i$. By \cite{ABS90} and \cite[Th.~5.2.2]{Car72}, we have $[U(i), U(j)] \subseteq U(i + j)$, and hence
\[
    \dl(U) \leq \lfloor\log_2(\max_{\alpha \in \Phi^+ \setminus \Phi_J^+} \operatorname{level}(\alpha)) \rfloor + 1 .
\]
Additionally, if $\sum_{\beta \in \Pi} d_\beta \beta$ is the root of $\Phi$ of largest height, then $\max_{\alpha \in \Phi^+ \setminus \Phi_J^+} \operatorname{level}(\alpha)$ is equal to $\sum_{\beta \in \Pi \setminus J} d_\beta$.
Hence the result can be proven by computing the root with largest height in $\Phi$, and using Remark~\ref{remark:sizeofJ} to determine the possibilities for $J$. For instance, if $G$ is of type $E_6$, then $\max_{\alpha \in \Phi^+ \setminus \Phi_J^+} \operatorname{level}(\alpha)) \leq 4$, with equality occurring if and only if $J$ is $\Pi$ with the second and fifth nodes removed. Hence, $\dl(U) \leq 2$ in this case.
\end{proof} 

By \cite[Th.~2.6.5(f)]{GLS94}, the subgroup $L_0 = \prod_{\alpha \in J} X_\alpha$ is a product of Chevalley groups and is normal in $L$. That is, $L_0 = G_1 \circ \cdots \circ G_m$ for some groups $G_i$ which are either quasisimple, or one of the following:
\[
    \SL_2(2), \> \SL_2(3), \> \SU_3(2), \> 
    B_2(2), \text{ or } {}^2B_2(2).
\]
In fact, each $G_i$ corresponds to a connected component of the Dynkin diagram of $\Phi_J$.

\begin{lemma} \label{lemLevi}
We always have $\dll(L) \leq 5$.
If $T$ is not ${}^2A_n(2)$ with $n$ even, then $\dll(L) \leq 4$.
If $q \geq 4$, then $\dll(L) \leq 1$.
\end{lemma}
\begin{proof}
If $q \geq 4$ then $L_0$ is a central product of quasisimple groups, and hence is perfect. Thus, since $L = L_0 . (K/K \cap L_0)$, and $K$ is abelian, we obtain $\dll(L) \leq 1$.
In general, if $L_0 = G_1 \circ \cdots \circ G_m$, then $\dll(G_i) \leq 4$, with equality if and only if $G_i = \SU_3(2)$. Hence, since $G_i$ can be equal to $\SU_3(2)$ only if $G$ is of type ${}^2A_n(q)$ with $n$ even, we have obtained the desired result.
\end{proof} 

\begin{remark}
Both $\dl(U)=3$ and $\dll(L)=5$ can be attained.
If $G$ is of type $E_8$, and $J$ is equal to
\[
    \dynkin [parabolic=8] E8 ,
\]
then the largest level of a root is 6. Additionally, we can use Magma \cite{magma} and \cite[Th.~5.2.2]{Car72} to determine that $[U(1), U(1)] = U(2)$ and $[U(2), U(2)] = U(4) \neq 1$.
Thus, $\dl(U) = 3$ in this case.
For $\dll(L)$, observe that $\PSU_7(2) (= {}^2A_6(2))$ contains a parabolic maximal subgroup with Levi factor $L=(\SL_2(4) \circ \SU_3(2)):3$; in particular, $\dll(L) = 5$.
\end{remark}

By combining Lemmas~\ref{lemUnip} and \ref{lemLevi}, we always obtain $\dll(T \cap H) \leq 7$. By Lemma~\ref{lemOut}, this gives $\dll(H) \leq 9$ if $T$ is not $\POmega_8^+(q)$ for $q$ odd.
If $T=\POmega_8^+(q)$, then we actually obtain $\dll(T \cap H) \leq 6$ and therefore $\dll(H) \leq 9$.
This concludes the proof of Theorem~\ref{thMain} for groups of Lie type.

\subsection{$T$ is sporadic}

Let $G$ be an almost simple group with sporadic socle $T$, and let $H$ be a maximal subgroup of $G$.
To read Table~\ref{table:Spor}, let
$$ \dllMax(G) \> := \> \max \{ \dll(H) : H \mbox{ is a maximal subgroup of } G \} , $$
so that Theorem~\ref{thMain} states that $\dllMax(G) \leq 10$ if $G$ is almost simple.

{\small
\begin{table}
	\centering
	\begin{tabular}{| c | c || c | c || c | c |}
		\hline
		
		\phantom{$\dfrac{1^1}{1_{1_1}}$}\hspace{-6mm}
		$T$ & $\dllMax(G)$ & $T$ & $\dllMax(G)$ & $T$ & $\dllMax(G)$ \\
		
		\hline \hline
		
		\phantom{$\dfrac{1^1}{1_{1_1}}$}\hspace{-6mm}
		$\mathrm{M_{11}}$ & $4$ & $\mathrm{McL}$ & $4$ & $\mathrm{Ly}$ & $2$ \\
		
		\phantom{$\dfrac{1^1}{1_{1_1}}$}\hspace{-6mm}
		$\mathrm{M_{12}}$ & $5$ & $\mathrm{He}$ & $4$ & $\mathrm{Th}$ & $7$ \\
		
		\phantom{$\dfrac{1^1}{1_{1_1}}$}\hspace{-6mm}
		$\mathrm{J_1}$ & $3$ & $\mathrm{Ru}$ & $4$ & $\Fi_{23}$ & $10$ \\
		
		\phantom{$\dfrac{1^1}{1_{1_1}}$}\hspace{-6mm}
		$\mathrm{M_{22}}$ & $1$ & $\mathrm{Suz}$ & $6$ & $\mathrm{Co_1}$ & $6$ \\
		
		\phantom{$\dfrac{1^1}{1_{1_1}}$}\hspace{-6mm}
		$\mathrm{J_2}$ & $4$ & $\mathrm{O'N}$ & $5$ & $\mathrm{J_4}$ & $6$ \\
		
		\phantom{$\dfrac{1^1}{1_{1_1}}$}\hspace{-6mm}
		$\mathrm{M_{23}}$ & $2$ & $\mathrm{Co_3}$ & $5$ & $\Fi_{24}'$ & $4$ \\
		
		\phantom{$\dfrac{1^1}{1_{1_1}}$}\hspace{-6mm}
		$\mathrm{HS}$ & $2$ & $\mathrm{Co_2}$ & $6$ & $\mathbb{B}$ & $6$ \\
		
		\phantom{$\dfrac{1^1}{1_{1_1}}$}\hspace{-6mm}
		$\mathrm{J_3}$ & $4$ & $\Fi_{22}$ & $6$ & $\mathbb{M}$ & between $4$ and $6$ \\
		
		\phantom{$\dfrac{1^1}{1_{1_1}}$}\hspace{-6mm}
		$\mathrm{M_{24}}$ & $2$ & $\mathrm{HN}$ & $6$ & $\mathrm{^2F_4(2)'}$ & $5$ \\
		
		\hline
		
	\end{tabular}
	\medskip{\caption{$\dllMax(G)$ for almost simple groups with sporadic socle.}
		\label{table:Spor}}
\end{table}
}

Let us describe how we obtain these values.

We use the $\Atlas$ \texttt{\MYhref[mypink]{https://brauer.maths.qmul.ac.uk/Atlas}{brauer.maths.qmul.ac.uk/Atlas}}
(see also the library \texttt{CTblLib (v1.3.9)} in \cite{GAP4}),
and we combine this data with Corollary~\ref{corExt} and Remark~\ref{remP}.
For instance, if $G = \mathbb{B}$ and $H$ is $[2^{30}].\L_5(2)$, then $\dl([2^{30}]) \leq 5$ and hence $\dll(H) \leq 5$.
More precisely, if $T$ is not one of $\Fi_{24}'$, $\mathbb{B}$ or $\mathbb{M}$, then the maximal subgroups of $G$ can be obtained by \cite{GAP4}, and thus $\dllMax(G)$ can be computed.
The equality in Theorem~\ref{thMain} is attained when $H = \texttt{AtlasSubgroup(``Fi23'',7)}$,
which is a solvable subgroup with shape $3^{1+8}.2^{1+6}.3^{1+2}.2S_4$.

If $T=\Fi_{24}'$, the maximal subgroups of $G=\Aut(T)$ can be obtained from the $\Atlas$. Thus, we can compute $\dll(H) \leq 4$ in this case (and $4$ is attained).
If a maximal subgroup of $T$ is equal to the intersection of a maximal subgroup of $G$ with $T$, then $\dll(H) \leq 4$ by Lemma~\ref{lemSub}.
Using the classification of the maximal subgroups of $\Fi_{24}'$ given in the $\Atlas$,
it is easy to bound $\dll(H) \leq 4$ in the remaining cases.

If $G=\mathbb{M}$, arguing as above we see that $\dll(H) \leq 6$ unless $H$ is $H_1 = 3^2.3^5.3^{10}.(\mathrm{M_{11}} \times \GL_2(3))$. In fact $H_1$ is available in Magma and we can compute $\dll(H_1) =4$.
Finding the exact value of $\dllMax(\mathbb{M})$ seems to be cumbersome with our current knowledge.

Finally, if $G=\mathbb{B}$, then $\dll(H) \leq 6$ unless $H$ is $H_2 = 3^2.3^3.3^6.(\Sym(4) \times \GL_2(3))$. From $H_1$ in $\mathbb{M}$, we obtain $H_2$ by centralizing and then factoring out an involution in the $\mathrm{M_{11}}$
(see \cite[Tab.~1]{AW04}).
This allows us to compute $\dl(H_2)=6$.

\vspace{0.1cm}
\section{Perfect groups} \label{sec4}

We now deal with primitive perfect groups.

\begin{proof}[Proof of Theorem~\ref{thPrimPerf}]
Let $n \geq 5$ and let $G < \Sym(n)$ be a primitive perfect group,
with $H<G$ being the stabilizer of a point.
Our goal is to prove that $\dll(H) \leq 10$.

Of course, we can apply Theorem~\ref{thMain} when $G$ is almost simple
(note that in this case $G$ is actually simple).
Now let $1 \neq N \unlhd G$, and observe that $NH=G$.
By Proposition~\ref{propGDL}, we have
\begin{align*}
\dll(H)  &  \> \leq \>
\dll(H/N \cap H) + \dll(N \cap H) \\ & \> = \>
\dll(G/N) + \dll(N \cap H) \> .
\end{align*}
Since $G$ and so $G/N$ is perfect, we obtain
\begin{equation} \label{eqPrimPerf}
\dll(H) \> \leq \> \dll(N \cap H) . 
\end{equation}

In particular, $H$ is perfect if $G$ has a regular normal subgroup (i.e. if $N \cap H =1$),
and Theorem~\ref{thPrimPerf} follows easily.
By the O'Nan-Scott theorem \cite[Sec.~4]{DM96}, we are left with the following cases:

\begin{itemize}

    \item[$\bullet$] \textbf{Product action.}
    There exists a primitive almost simple group $X$ with socle $T$ and stabilizer $Y$, and a transitive group $S \leqslant \Sym(d)$ such that
    $$ T^d \unlhd G \leqslant X \wr_d S . $$
    (Note that $G$ acts on $|X:Y|^d$ points).
    Let $T^d \cong N \unlhd G$.
    We have $H = Y^d S \cap G$,
    and so $N \cap H = N \cap Y^d \cong (T \cap Y)^d$.
    By (\ref{eqPrimPerf}) we have
    $$ \dll(H) \leq \dll(N \cap H) = \dll(T \cap Y) . $$
    As discussed we have $Y^{(3)} \subseteq T \cap Y$,
    so by Lemma~\ref{lemSub} we obtain
    $$ \dll(T \cap Y) \leq \dll(Y) \leq 10 $$
    as desired.
    
    \item[$\bullet$] \textbf{Diagonal type.}
    Let $N \cong T^k$ be the socle of $G$.
    If $G$ is of simple diagonal type, then $N \cap H$ is isomorphic to $T$ and embeds subdirectly in $T^k$.
    If $G$ is of compound diagonal type, then $k= \ell m$ ($\ell,m \geq 2$) and $N \cap H \cong T^\ell$ is a product of $\ell$ diagonals of $T^m$.
    In any case, $N \cap H$ is perfect and we can use (\ref{eqPrimPerf}) to conclude that $H$ is perfect.
    
\end{itemize}

The proof of $\dll(H) \leq 10$ is complete.

We finally show that $\dll(H) =10$ cannot happen in product action. By the discussion above, the only possibility would be with $X=T=\Fi_{23}$ and $Y<X$ the solvable maximal subgroup with $\dl(Y) = 10$. Then $G=X \wr_d S$ for some transitive perfect group $S < \Sym(d)$ and $H = Y \wr_d S$. However, by Lemma~\ref{lem:wreath'} we have that $H'$ is perfect and $\dll(H) =1$ in this case.
\end{proof} 

We would like to point out the following result for quasisimple groups,
i.e. perfect central extensions of non-abelian simple groups.
For completeness, we also discuss the minimum
number of generators, which we denote by $d(G)$.

\begin{proposition} \label{propQS}
Let $G$ be a quasisimple group and let $H$ be a maximal subgroup of
$G$. Then $\dll(H) \leq 10$, and $d(H) \leq 5$.
\end{proposition}
\begin{proof}
We first observe that $Z = Z(G) \subseteq H$.
In fact, otherwise we would have
$$ G = G' = (HZ)' = H' \neq G. $$
Let $G/Z \cong T$, and let us write $\mathrm{M}(T)$ for the Schur multiplier of $T$.
Now $H/Z$ is isomorphic to a maximal subgroup of $T$, so $\dll(H) \leq \dll(H/Z) + \dl(Z)$ by Proposition~\ref{propGDL},
and $\dll(H) \leq 10$ follows by Theorem~\ref{thMain}
and the fact that $\mathrm{M}(\Fi_{23})$ is trivial.

For the number of generators, clearly we have $d(H) \leq d(H/Z) + d(Z)$,
and $d(H) \leq 6$ follows by \cite{BLS13} and the fact that $\mathrm{M}(T)$ can be generated by at most two elements (see the $\Atlas$).
Looking at \cite[Tab.~1]{LMT20},
we see that the only possibilities for $d(H) = 6$ to happen are with
$T = \POmega^+_n(q)$ ($\mathrm{M}(T)$ should be non-cyclic).
However, it is easy to see that $d(H) = 4$ in these cases.
\end{proof}

We do not know whether the bound for $d(H)$ in Proposition~\ref{propQS} can be improved to $4$.

To conclude the paper, we make a comment on the constant in Conjecture~\ref{conjTwoPoints},
i.e. what can be achieved by stabilizing two points in an arbitrary primitive group.
For $G=\Fi_{23}$ and the solvable maximal subgroup $H$ with $\dl(H) =10$, we can find $g \in G$ such that $\dl(H \cap H^g) =2$.
However, if $G=\Aut(\POmega_8^+(3))$ and $H$ is the solvable maximal subgroup of index $36\,400$, then the corresponding primitive group has rank $5$ and the derived lengths of stabilizers of two points are $[8,6,6,5,5]$. In particular, $\dl(H \cap H^g) \geq 5$ for all $g \in G$.
We can do better with groups of product type:

\begin{example} \label{exBad2}
Let $X = \Aut(\POmega_8^+(3))$ and let $Y<X$ be the solvable maximal subgroup discussed above.
Let $S = (\Sym(4) \wr_4 \Sym(4)) \wr C_2$ act on $32$ points,
and let $G = X \wr_{32} S$.
Then $H = Y \wr_{32} S$ is a solvable maximal subgroup of $G$.
For $g = (v,s) \in G$, since $Y^d S$ is normalized by $s$, we have
$$ H \cap H^g = Y^dS \cap (Y^dS)^{vs} = (Y^d S \cap (Y^d S)^v)^s , $$
and it follows that $\dl(H \cap H^g) = \dl(H \cap H^v)$.
Now we argue as in \cite[Sec.~3.2]{Sab26}.
All set-stabilizers of $S$ have derived length at least $3$ \cite[Ex.~2.11]{Sab26}, and there is no $4$-coloring with trivial stabilizer (because already $\Sym(4) \wr_4 \Sym(4)$ on $16$ points has none).
These facts together imply that $\dl(H \cap H^g) \geq 7$ for all $g \in G$.
\end{example}

\vspace{0.3cm} \noindent
{\bfseries Acknowledgments:}
We thank the referee for their helpful comments,
which significantly improved the presentation of the paper.

   \vspace{0.5cm}

\end{document}